\documentclass[11pt, reqno]{amsart}

\usepackage{ifpdf}
\ifpdf
  \usepackage[hyperindex,pagebackref]{hyperref}%
\else
  \expandafter\ifx\csname dvipdfm\endcsname\relax
    \usepackage[hypertex,hyperindex,pagebackref]{hyperref}
  \else
    \usepackage[dvipdfm,hyperindex,pagebackref]{hyperref}
  \fi
\fi

\usepackage{cases}
\usepackage[square, comma, sort&compress, numbers]{natbib}
\usepackage{float}
\usepackage{pifont}
\usepackage{bbding}
\usepackage{amssymb}
\usepackage{mathrsfs}
\usepackage{subfigure}
\usepackage{caption}
\usepackage{geometry}
\usepackage{amsmath, amsthm, amscd, amsfonts, amssymb, graphicx, color}
\usepackage{lineno}

\usepackage{ifpdf}

\newtheorem{theorem}{Theorem}

\newtheorem{lemma}{Lemma}

\newtheorem{rem}{Remark}

\numberwithin{equation}{section}

\newcommand{\abs}[1]{\left\vert#1\right\vert}

\newcommand{\R}{\mbox{$\mathbb{R}$}}

\newcommand{\C}{\mbox{$\mathbb{C}$}}

\newcommand{\D}{\mbox{$\mathbb{D}$}}

\makeatletter
\@namedef{subjclassname@2020}{\textup{2020} Mathematics Subject Classification}
\makeatother

\date{\today}
\begin{document}
\setcounter{page}{1}

	\title[On Koebe-type functions for harmonic quasiconformal mappings]
	{On Koebe-type functions for harmonic quasiconformal mappings}
	
	\author[Z.-G. Wang, J.-L. Qiu and A. Rasila]{Zhi-Gang Wang, Jia-Le Qiu and Antti Rasila}

		\address{\noindent Zhi-Gang Wang\vskip.01in
		School of Mathematics and Statistics, Hunan First Normal University, Changsha 410205, Hunan, P. R. China.}
	\email{\textcolor[rgb]{0.00,0.00,0.84}{wangmath$@$163.com}}

	\address{\noindent Jia-Le Qiu \vskip.01in
		School of Mathematics and Statistics, Changsha University of Science and Technology,
		Changsha 410114, Hunan, P. R. China.}
	\vskip.01in
        \address{\noindent Hunan Provincial Key Laboratory of Mathematical Modeling and Analysis in
Engineering, Changsha University of Science and Technology,
		Changsha 410114, Hunan, P. R. China.}
	\email{\textcolor[rgb]{0.00,0.00,0.84}{qiujiale2023$@$163.com}}

\address{\noindent Antti Rasila \vskip.01in
	Department of Mathematics with Computer Science, Guangdong Technion-Israel Institute
of Technology, 241 Daxue Road, Shantou 515063, Guangdong, P. R.
	China.}

\address{\noindent
Department of Mathematics, Technion-Israel Institute of Technology, Haifa 3200003, Israel.}
\email{\textcolor[rgb]{0.00,0.00,0.84}{antti.rasila$@$iki.fi}; \textcolor[rgb]{0.00,0.00,0.84}{antti.rasila$@$gtiit.edu.cn}}

\thanks{2020 \textit {Mathematics Subject Classification.} Primary 31A05; Secondary 30C55, 30C62.}

\thanks{\textit {Keywords and Phrases.} {Generalized Koebe function; harmonic quasiconformal mapping; pre-Schwarzian and Schwarzian norms.}}
	

	
	
	\date{\today}

	\begin{abstract}
This paper studies a class of Koebe-type harmonic quasiconformal functions. It is motivated by the shear construction of Clunie and Sheil-Small [Ann. Acad. Sci. Fenn. Ser. A I Math. 9: 3--25, 1984] and the harmonic quasiconformal Koebe function. Equivalent univalence conditions, pre-Schwarzian and Schwarzian norms, coefficient inequalities, as well as growth and area theorems for this family of functions are established. These findings improve several previously known results.
	\end{abstract} \maketitle

	
	\section{Introduction and preliminaries}
	Let $\mathcal S$ be the family of all univalent analytic functions $\varphi$ in the open unit disk $\D$ with the normalizations $\varphi(0)=\varphi'(0)-1=0$.
The extremal functions for the class $\mathcal S$ are the Koebe function
\begin{equation*}\label{3}
    k(z):=\dfrac{z}{(1-z)^2}\quad(z\in\D)
\end{equation*} and its rotations.
The Koebe function $k(z)$ is the extremal function of the Bieberbach conjecture (now known as the de Branges theorem).
Geometrically, it maps $\D$ onto the complex plane minus the segment of the negative real axis from $-1/4$ to infinity.

The generalized Koebe function $k_a(z)$ is defined by
\begin{equation}\label{4}
    k_a(z):=\dfrac{1}{2a}\left [\left(\dfrac{1+z}{1-z}\right)^a-1\right]\quad\left(z\in\D;\,a\in \C\backslash\{0\}\right),
\end{equation}
which coincides with the classical Koebe function $k(z)$ for $a=2$. 
The function $k_a(z)$ serves as an extremal function for several interesting problems, see e.g., \cite{ga,kt,m,nst,p}. 
By noting that if we take $a\rightarrow0$ in \eqref{4}, the function $k_a(z)$ reduces to 
\begin{equation*}\label{5}
    k_0(z):=\dfrac{1}{2}\log\dfrac{1+z}{1-z}\quad(z\in\D).
\end{equation*}

Over the years, various generalizations of Koebe functions have been introduced in geometric function theory (cf. \cite{dy,hm1,hrz,ma,pr}).	

Let $\mathcal H$ denote the class of complex-valued harmonic functions $f=h+\overline{g}$ in $\D$,
normalized by the conditions $f(0)=f_{z}(0)-1=0$, which have the
form
\begin{equation}\label{111}
f(z)=z+\sum_{n=2}^{\infty}a_nz^n+\overline{\sum_{n=1}^{\infty}b_nz^n}.
\end{equation}
 The Jacobian of $f$ is given by
\begin{equation*}\label{13}
    J_f(z)=|h'(z)|^2-|g'(z)|^2.
\end{equation*}
In \cite{l}, Lewy showed that a harmonic function $f=h+\overline{g}\in\mathcal H$ is locally univalent and sense-preserving if and only if $J_f(z)>0$. This statement is equivalent to $|\omega(z)|=\abs{g'(z)/h'(z)}<1$ with $h'(z)\neq 0$. The quantity $\omega(z)$ is called the complex dilatation of $f$.

In 2015, Hern\'andez and Mart\'in \cite{hm2} introduced the pre-Schwarzian derivative $P_f$ and Schwarzian derivative $S_f$ for a locally univalent harmonic mapping:
\begin{equation*}\label{9}
       P_f(z):=\dfrac{h''(z)}{h'(z)}-\dfrac{\omega'(z)\overline {\omega(z)}}{1-|\omega(z)|^2}
\end{equation*}
and
\begin{equation*}\label{10}
 S_f(z):=\dfrac{h'''(z)}{h'(z)}-\dfrac{3}{2}\left(\dfrac{h''(z)}{h'(z)}\right)^2
 +\dfrac{\overline {\omega(z)}}{1-|\omega|^2}\left(\dfrac{h''(z)}{h'(z)}\omega'(z)-\omega''(z)\right)-\dfrac{3}{2}\left(\dfrac{\omega'(z)\overline {\omega(z)}}{1-|\omega|^2}\right)^2.
 \end{equation*}   
We observe that $P_f$ and $S_f$ are generalizations of the pre-Schwarzian and Schwarzian derivatives of analytic functions.
 The corresponding pre-Schwarzian and Schwarzian norms of a locally univalent harmonic function $f$ are defined as follows:
 \begin{equation*}\label{11}
     \|P_f\|:=\sup_{z\in{\scriptsize \D}}|P_f(z)|\left(1-|z|^2\right)
 \end{equation*} and
 \begin{equation*}\label{12}
     \|S_f\|:=\sup_{z\in{\scriptsize \D}}|S_f(z)|\left(1-|z|^2\right)^2.
 \end{equation*}
For recent developments on pre-Schwarzian and Schwarzian norms of harmonic mappings, see \cite{
cp,chm,g2,
wwrq}.

\subsection{Shearing method}
In 1984, Clunie and Sheil-Small \cite{cs} (see also \cite{du}) proved the following classical result by using \textit{shearing method}.
\begin{lemma}\label{lem1}
Let $f=h+\overline{g}$ be a locally univalent harmonic mapping in $\D$. Then it is univalent and convex in the  direction $\theta$ if and only if the analytic function $h-e^{2i\theta}g$ is univalent and convex in the direction $\theta$. 
\end{lemma}

For further details regarding the shear construction, we refer the reader to \cite{fhm,kpv,mpq,wang}.
Moreover, for a given analytic function $\phi$ convex in the direction $\theta$ and a prescribed dilatation $\omega$, the above shear construction provides a method to construct univalent harmonic mappings. 
A classical example of univalent harmonic mappings concerning this method is the harmonic Koebe function 
\begin{equation}
    {K}(z)=H+\overline{G}=\dfrac{z-({1}/{2})z^2+({1}/{6})z^3}{(1-z)^3}+\overline {\dfrac{({1}/{2})z^2+({1}/{6})z^3}{(1-z)^3}},
\end{equation}
where $H$ and $G$ with $H(0)=G(0)=0$ are solutions to the system of equations 
\begin{equation*}\label{15}
    \begin{cases}
        H(z)-G(z)=k(z),\\
      \,  {G'(z)}/{H'(z)}=z.
    \end{cases}
\end{equation*}


In 2016, Ferrada-Salas and Mart\'in \cite{fm} used the shearing method  to construct a family $$\mathcal {K_H}(\nu,a,\mu,R):=\{f=h+\overline{g}:\, a\in\C;\,|\nu|=|\mu|=1;\,0\leq R\leq1\}$$ of generalized harmonic Koebe functions, which are sheared by
the system
\begin{equation*}\label{8}
    \begin{cases}
        h(z)-\nu g(z)=k_a(z),\\
        \ \ \, g'(z)/h'(z)=\mu\, l_R(z),
    \end{cases}
\end{equation*}
where the lens-map $l_R(z)$ is given by
\begin{equation*}\label{7}
    l_R(z)=\dfrac{\left[{(1+z)}/{(1-z)}\right]^R-1}{\left[{(1+z)}/{(1-z)}\right]^R+1}\quad(z\in\D;\,0\leq R\leq1).
\end{equation*}
They obtained 
several interesting properties for the family $\mathcal {K_H}(\nu,a,\mu,R)$.

\subsection{Harmonic quasiconformal Koebe function}

We say that $f$ belongs to the class $\mathcal{S}_\mathcal{H}(K)$ of \textit{harmonic $K$-quasiconformal mappings}, where $K\ge 1$ is a constant, if $f=h+\overline{g}$ is an univalent harmonic mapping and its dilatation satisfies the condition 
$\abs{{g'}/{h'}}\leq \lambda$, where $\lambda\in [0,1)$ is given by
$$
\lambda:=\frac{K-1}{K+1}\quad(K\geq 1).
$$ A function $f$ is called a harmonic quasiconformal mapping, if it belongs to $\mathcal{S}_\mathcal{H}(K)$ for some $K\ge 1$ (cf. \cite{lz,wang1}).

Wang \textit{et al.} \cite{wwrq} in 2024 constructed the so-called harmonic $K$-quasiconformal Koebe function \begin{align*}\begin{split}\label{31}
  f_\lambda(z)&=\frac{1}{(\lambda-1)^3}{\left[\frac{(\lambda-1)(1-3\lambda+2\lambda z)z }{(1-z)^2}+\lambda(\lambda+1) \log\left(\frac{1-z}{1-\lambda z}\right)\right]}\\ & \qquad\ \ \ +\frac{\lambda}{(\lambda-1)^3}\overline{{\left[\frac{(1-\lambda)(1+\lambda-2z)z}{(1-z)^2}+(\lambda+1) \log\left(\frac{1-z}{1-\lambda z}\right)\right]} },\end{split}
\end{align*}
 which is sheared by the system
\begin{equation*}
    \begin{cases}
       h(z)-g(z)=k(z),\\
\ g'(z)/h'(z)=\lambda z\quad(0\leq \lambda<1).
    \end{cases}
\end{equation*}
It looks tempting to assume that $f_\lambda(z)$ is an extremal function for the family of harmonic $K$-quasiconformal mappings. They subsequently posed a series of conjectures involving the class $\mathcal{S}_\mathcal{H}(K)$ (cf. \cite{wwrq}).
Noting that for $\lambda=0$, we get the classical Koebe function, and for $\lambda\rightarrow1^-$, the function coincides exactly with the harmonic Koebe function.

It is worth mentioning that Li and Ponnusamy \cite{lp} later verified that the above function is indeed extremal for several subclasses of harmonic quasiconformal mappings. Building on their work, Das, Huang, and Rasila \cite{dhr} have further extended the corresponding results obtained by Li and Ponnusamy in \cite{lp}.

\subsection{Koebe-type harmonic quasiconformal functions}
Motivated by the class $\mathcal {K_H}(\nu,a,\mu,R)$ and the harmonic quasiconformal Koebe function $f_\lambda(z)$, for the system of differential equations given by
\begin{equation}\label{1}
  \begin{cases}
      h(z)-g(z)=k_a(z)\quad(a\in\R),\\
     \,{g'(z)}/{h'(z)}=\lambda z\quad(0\leq\lambda<1),
  \end{cases}  
\end{equation}
we construct a class of Koebe-type harmonic quasiconformal functions 
${f}_{a,\lambda}=h+\overline{g}$,
where 
$$h(z):=-\frac{-2 \lambda  \, _2F_1\left(1,-a;1-a;\frac{\lambda +1}{\lambda -1}\right)+\lambda -1}{2 a \left(\lambda ^2-1\right)}-\frac{\psi(a,\lambda,z)}{4a(a-1)(\lambda-1) (\lambda +1)}$$
with
$$\psi(a,\lambda,z):=\left(\frac{1}{1-z}\right)^a \bigg\{4 (a-1) \lambda  (z+1)^a \, _2F_1\left(1,-a;1-a;-\frac{(z-1) (\lambda +1)}{(z+1) (\lambda -1)}\right)$$$$+2^a (\lambda -1) \bigg[a (z-1) \, _2F_1\left(1-a,1-a;2-a;\frac{1-z}{2}\right)$$$$-2 (a-1) \, _2F_1\left(-a,-a;1-a;\frac{1-z}{2}\right)\bigg]\bigg\},$$
and $$g(z):=\frac{\varpi(a,\lambda,z)}{2 a \left(\lambda ^2-1\right)}$$
with \begin{align*}
\begin{split}
\varpi(a,\lambda,z&):=\lambda  \bigg\{2 \, _2F_1\left(1,-a;1-a;\frac{\lambda +1}{\lambda -1}\right)\\&-\left(\frac{1+z}{1-z}\right)^a \left[2\,_2F_1\left(1,-a;1-a;-\frac{(z-1) (\lambda +1)}{(z+1) (\lambda -1)}\right)+\lambda -1\right]+\lambda -1\bigg\}.
\end{split}
\end{align*}
As in the classical case, the function ${f}_{2,\lambda}(z)$ coincides with the harmonic quasiconformal Koebe function ${f}_{\lambda}(z)$.
We denote the class of Koebe-type harmonic quasiconformal functions by $$\mathcal{K}_{a,\lambda}:=\{{f}_{a,\lambda}=h+\overline{g}\in\mathcal{H}: a\in\R;\, \lambda\in[0,1)\}.$$

We present the figures of $f_{0,0}(z)$, $f_{0,1/2}(z)$, $f_{2,0}(z)$, $f_{2,1/2}(z)$, $K(z)$ and $f_{3,1/2}(z)$ to illuminate 
the family $\mathcal{K}_{a,\lambda}$ (see Figure \ref{F1}).

\begin{figure}[H] 
\centering 
\subfigure[Image of $f_{0,0}(z)$.]{ 
\label{Fig.sub.1} 
\includegraphics[height=5.0cm]{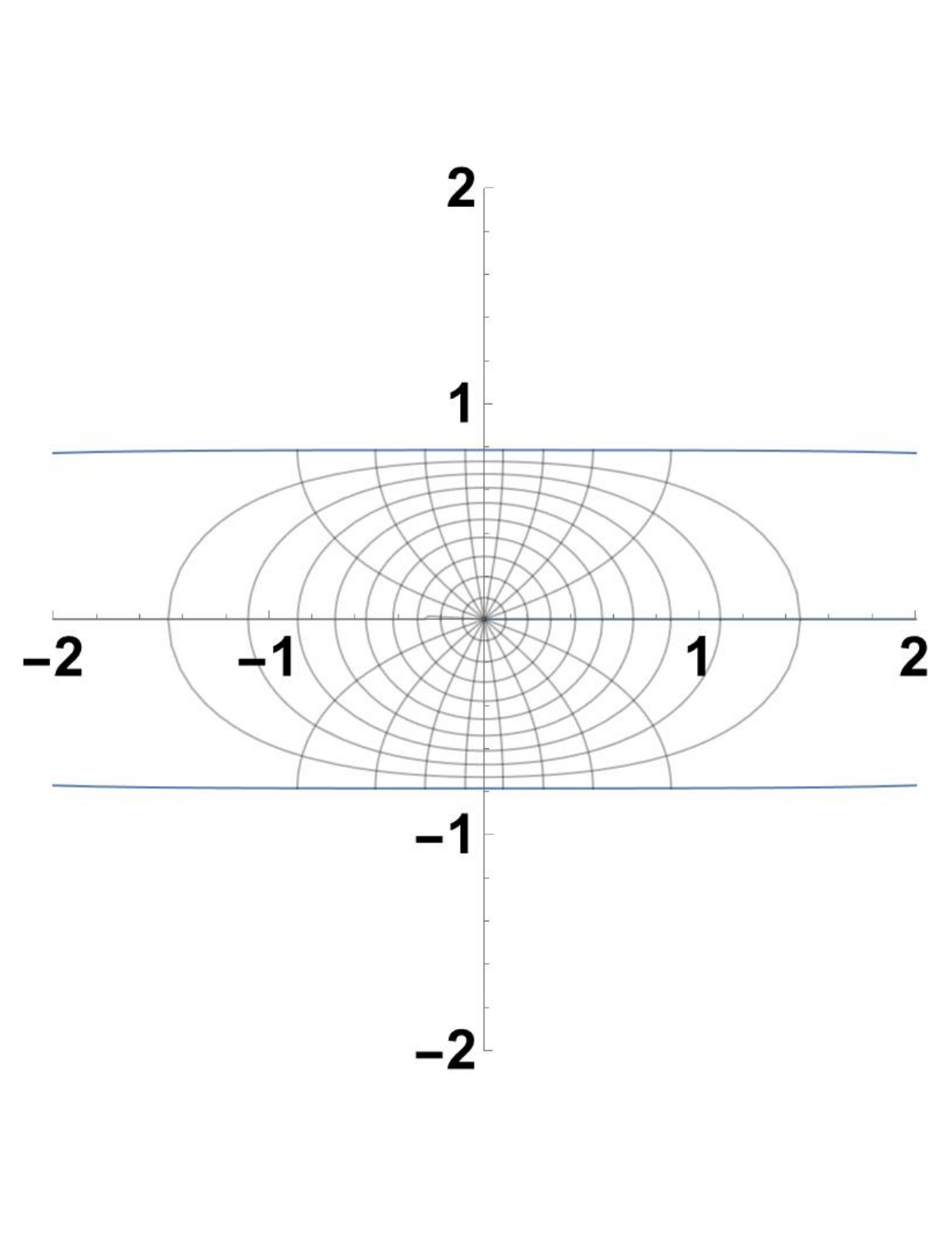}} 
\quad
\subfigure[Image of $f_{0,1/2}(z)$.]{ 
\label{Fig.sub.2} 
\includegraphics[height=5.0cm]{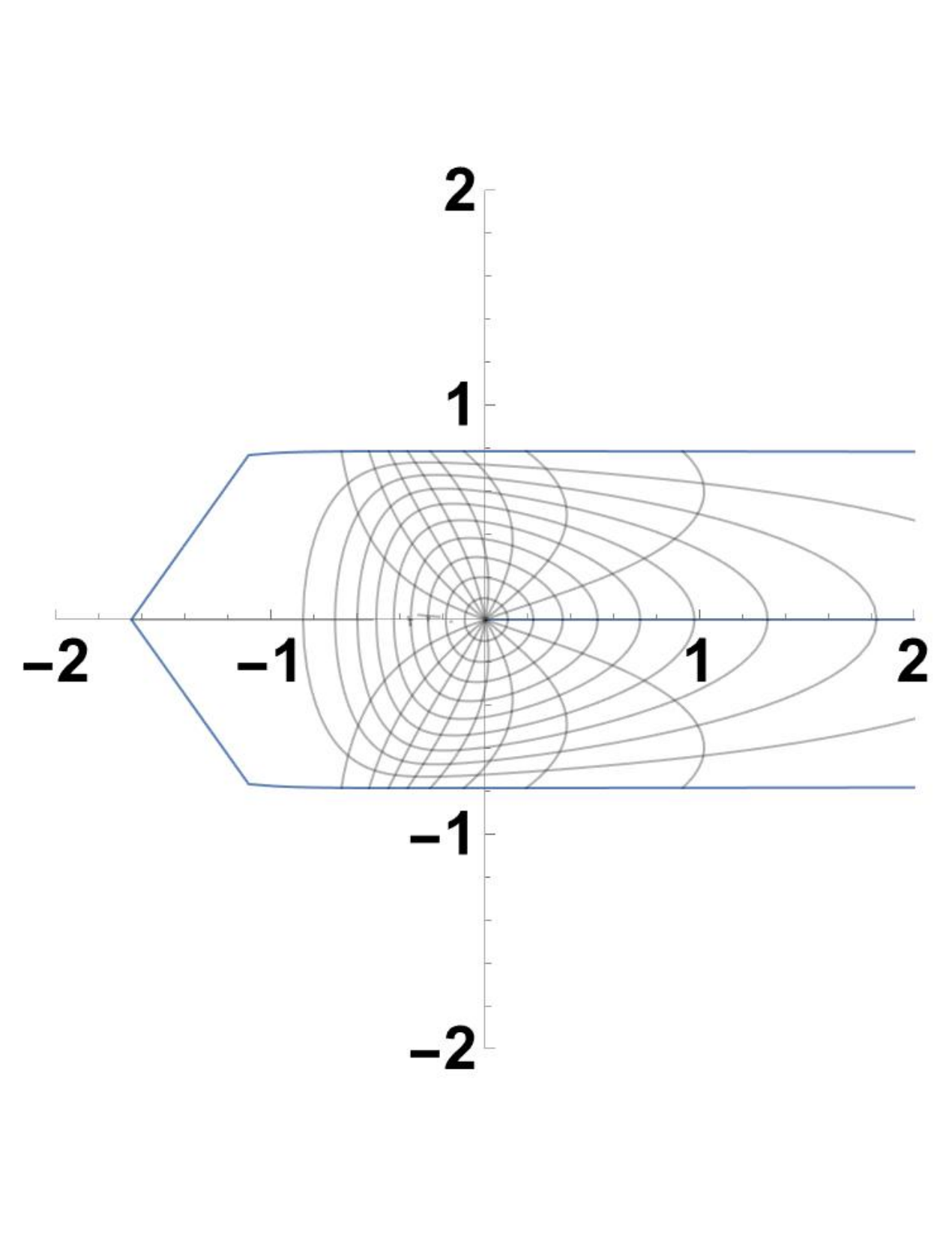}} 
\end{figure}

\begin{figure}[H] 
\centering 
\subfigure[Image of $f_{2,0}(z)$.]{ 
\label{Fig.sub.3} 
\includegraphics[height=5.0cm]{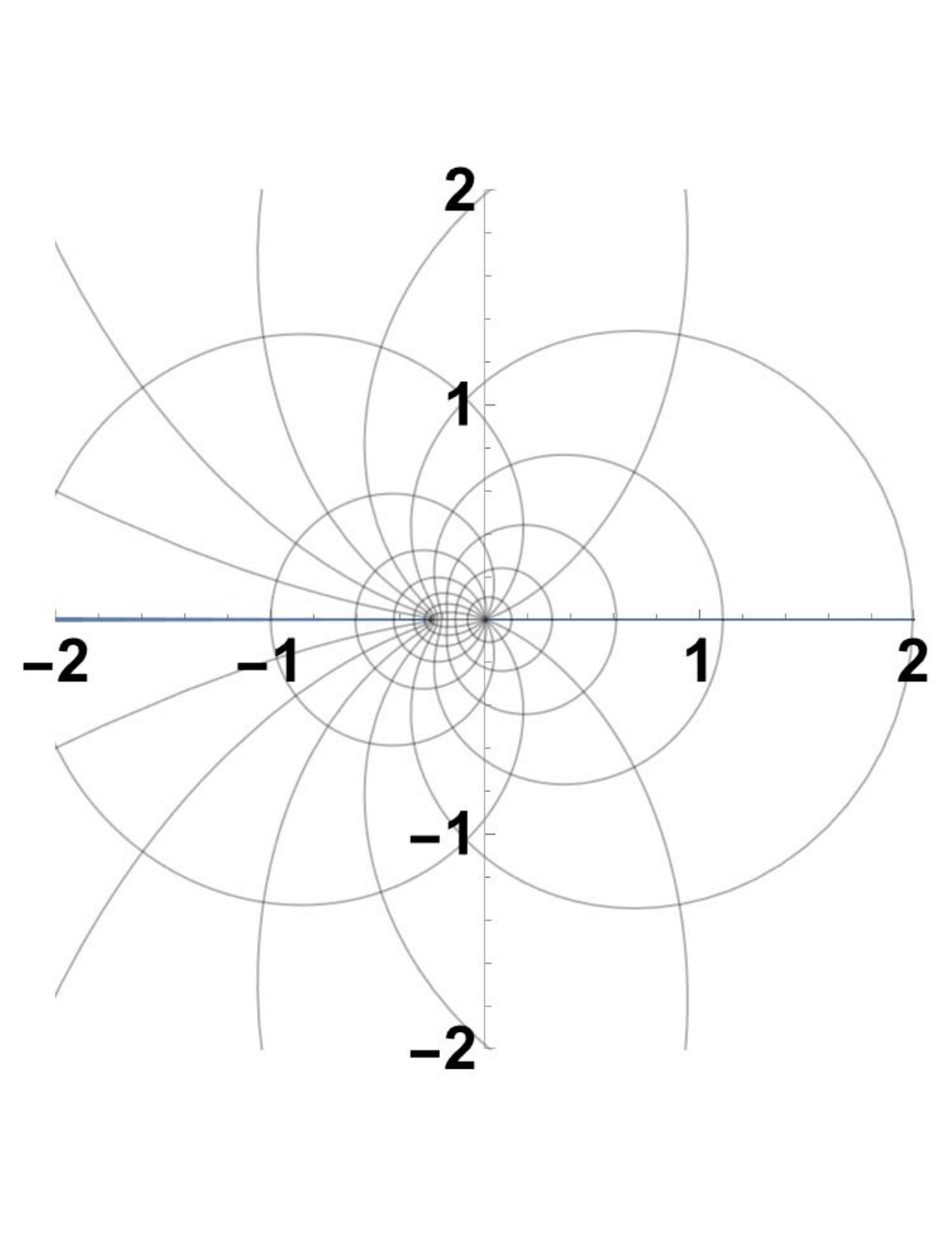}} 
\quad
\subfigure[Image of $f_{2,1/2}(z)$.]{ 
\label{Fig.sub.4} 
\includegraphics[height=5.0cm]{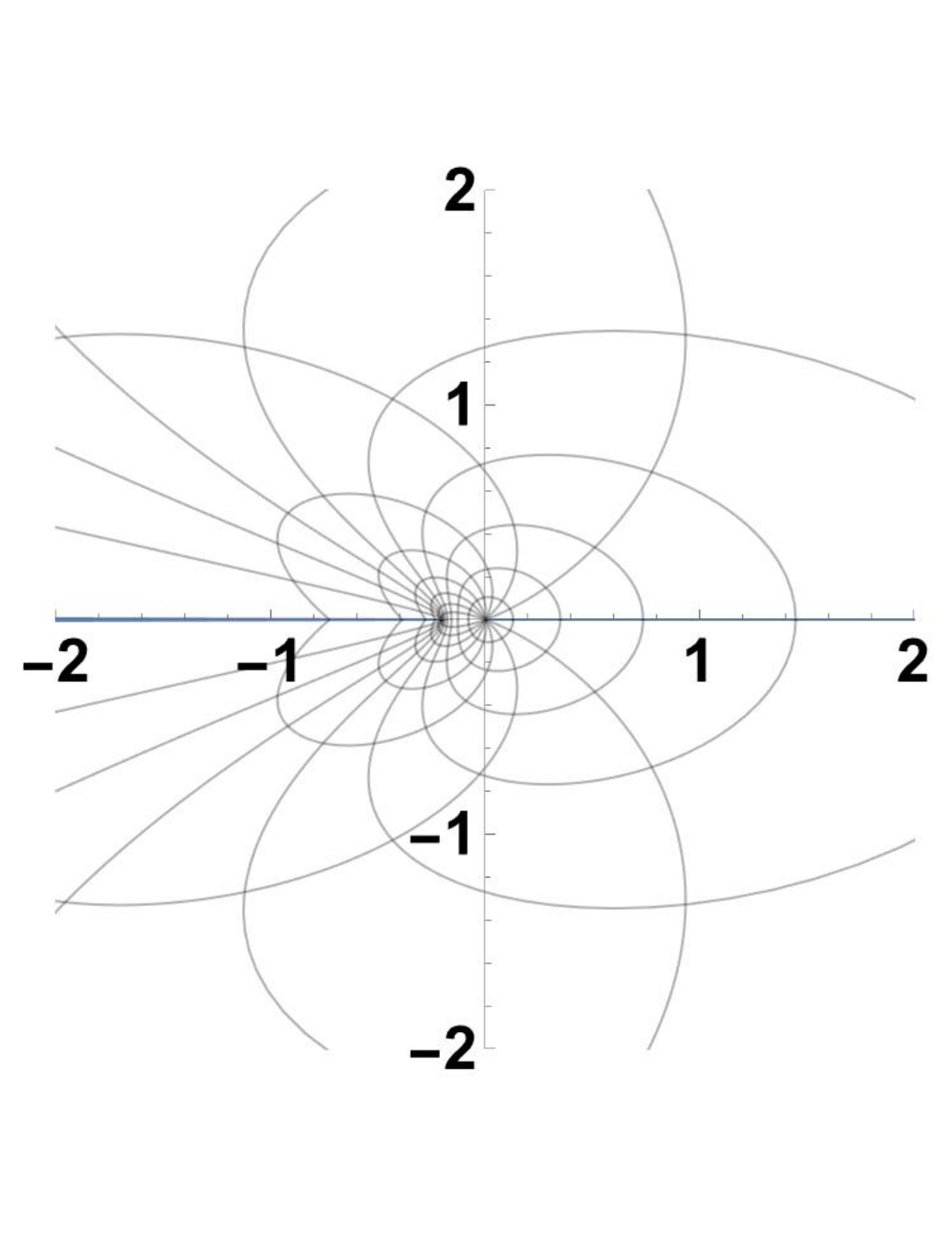}} 
\end{figure}

\begin{figure}[H] 
\centering 
\subfigure[Image of ${K}(z)$.]{ 
\label{Fig.sub.5} 
\includegraphics[height=5.0cm]{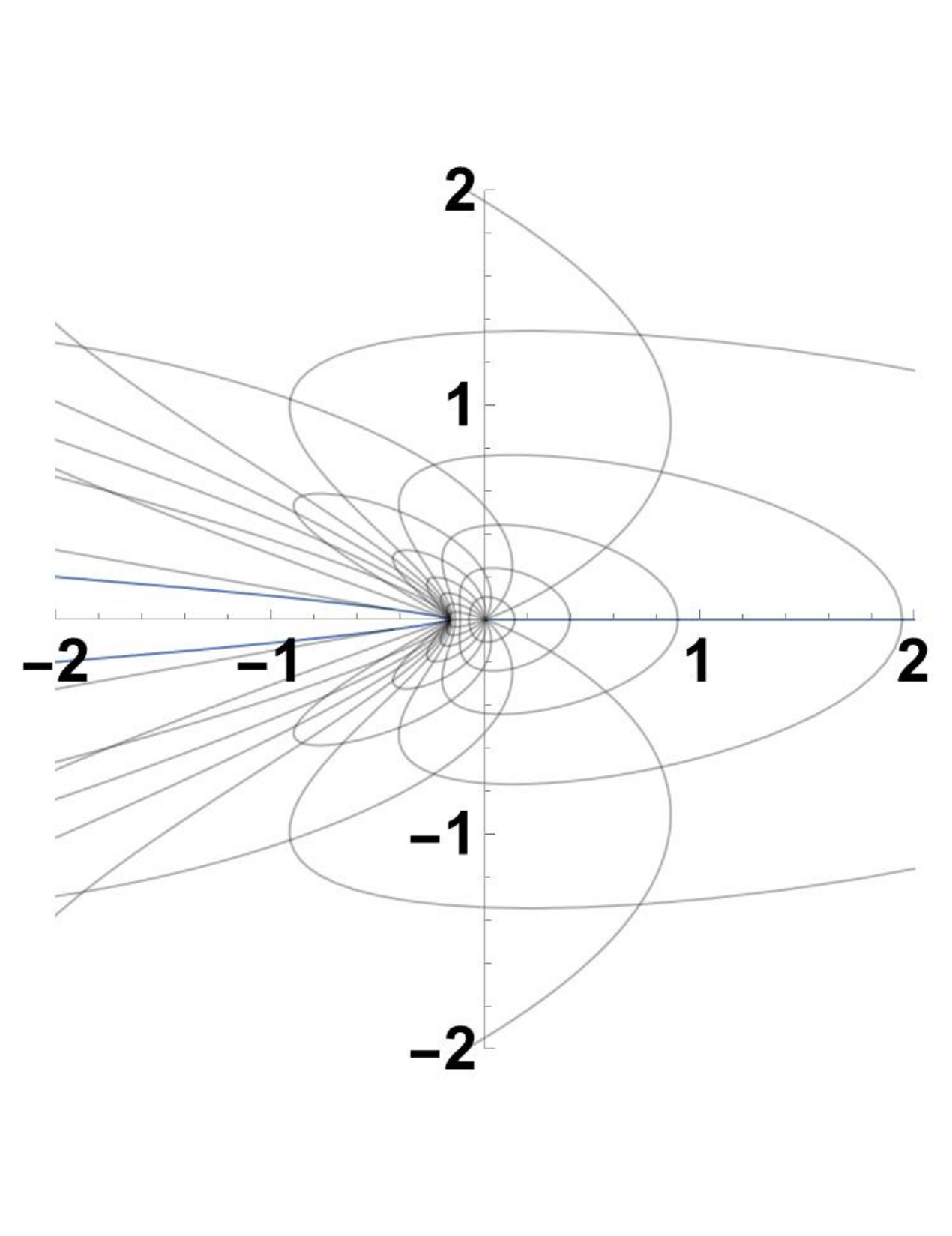}} 
\quad
\subfigure[Image of $f_{3,1/2}(z)$.]{ 
\label{Fig.sub.6} 
\includegraphics[height=5.0cm]{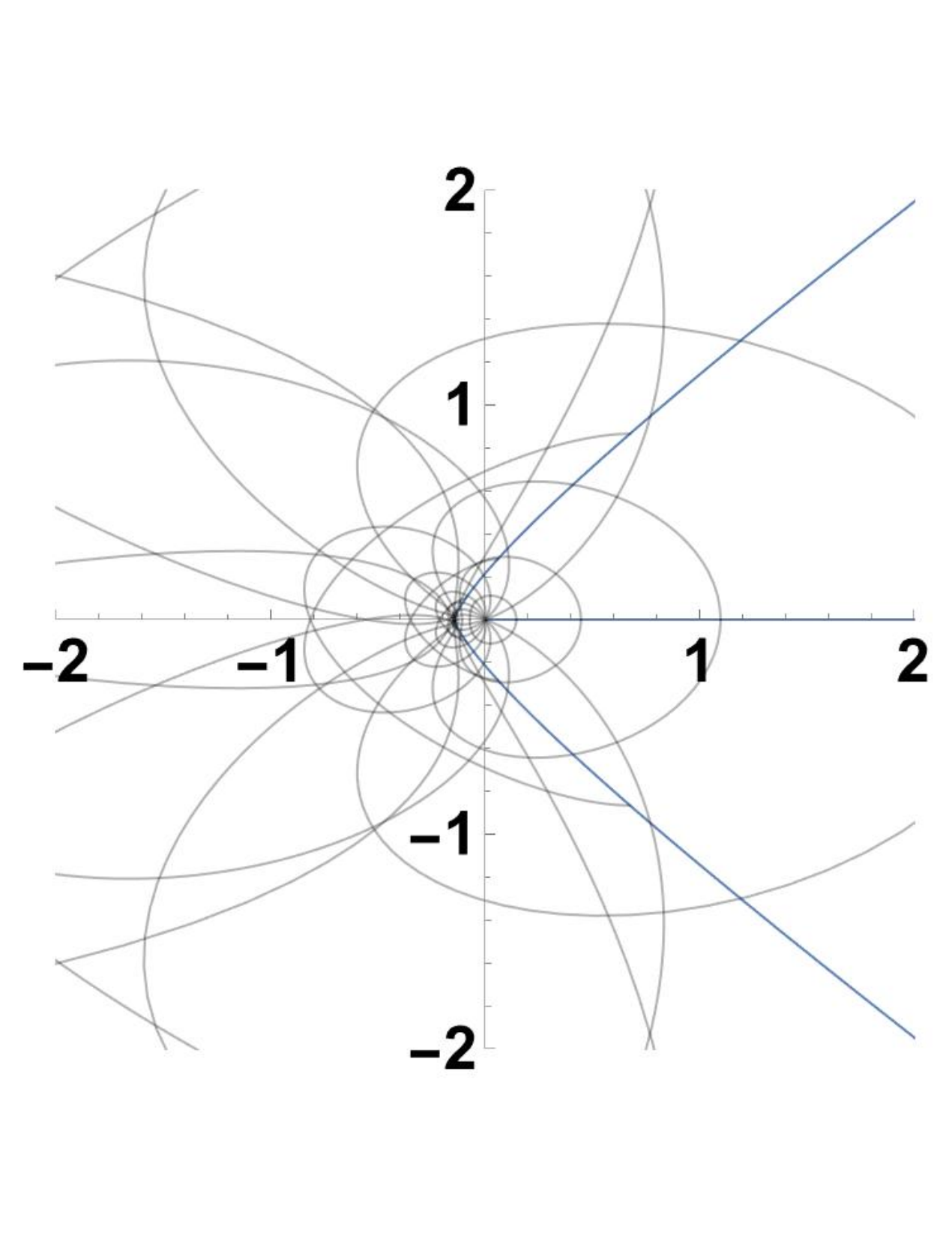}} 
\caption{\small{Images of the unit disk under the mappings $f_{0,0}(z)$, $f_{0,1/2}(z)$, $f_{2,0}(z)$, $f_{2,1/2}(z)$, ${K}(z)$
and $f_{3,1/2}(z)$.}}
\label{F1} 
\end{figure}

The primary objective of this paper is to establish equivalent univalence conditions, pre-Schwarzian and Schwarzian norms, coefficient inequalities, as well as growth and area theorems for the class $\mathcal{K}_{a,\lambda}$
of Koebe-type harmonic quasiconformal functions.


	\vskip.10in
    \section{Properties of the family $\mathcal{K}_{a,\lambda}$}
	\subsection{Equivalent univalence conditions for the family $\mathcal{K}_{a,\lambda}$}
	Firstly, we derive  univalence conditions for the family $\mathcal{K}_{a,\lambda}$.


	\begin{theorem}\label{T1}
	Let $a\in\R$ and $\lambda\in[0,1)$. Then the function $f_{a,\lambda}$ is univalent if and only if  $-2\leq a\leq2$. 
\end{theorem}
\begin{proof}
	We first consider the case $a=0$. Then $f_{0,\lambda}=h+\overline{g}$ satisfies the conditions  
	\begin{equation*}\label{T1-1}
		\begin{cases}
			h(z)-g(z)=k_0(z),\\
			\,{g'(z)}/{h'(z)}=\lambda z.
		\end{cases}
	\end{equation*}
Noting that $k_0(z)$ maps $\D$ onto a convex domain,
   then by Lemma \ref{lem1}, it follows that $f_{0,\lambda}(z)$ is univalent. 
    
   Next, suppose that $0<a\leq 2$. Then $k_a(z)$ maps $\D$ onto a region convex in the direction of the real axis.  Again, by applying Lemma \ref{lem1}, we conclude that $f_{a,\lambda}$ is univalent for $a\in(0,2]$. 
    
	Now, we prove that $f_{a,\lambda}$ is non-univalent  for $a>2$.  Note that $k_a(z)$ and $\lambda z$ have real coefficients. By \eqref{1}, we know that $h'(z)-g'(z)=k_a'(z)$, and ${g'(z)}/{h'(z)}=\lambda z$. Thus, the coefficients of $h$ and $g$ are also real numbers. 
    Consider the transformation 
    $$F(z)=\dfrac{1+z}{1-z},$$
    which maps the unit disk onto the right half-plane. 
    So there exists $z_1\in\D$ such that
    \begin{equation*}\label{T1-2}
        \dfrac{1+z_1}{1-z_1}=e^{i\frac{\pi}{a}}\quad(a>2),
    \end{equation*}
    which implies that $z_1$ is not a real number in $\D$.
    Therefore, we have 
    \begin{equation}\label{T1-3}
        k_a(z_1)=-\dfrac{1}{a}\quad(a>2).
    \end{equation}
    
 Suppose that $z_2=\overline{z_1}$. Since $k_a$ has real coefficients, by \eqref{T1-3}, it is easy to check that
   \begin{equation}\label{211}
       k_a(z_1)=\overline {k_a(z_1)}=k_a(\overline{z_1})=k_a(z_2).
   \end{equation}
   Note that  \begin{equation}\label{212}h(z)-g(z)=k_a(z).\end{equation}
Combining \eqref{211} and \eqref{212}, we get
    $$h(z_1)-g(z_1)=h(z_2)-g(z_2).$$ 
    By observing that the function $g$ also has real coefficients, it follows that  $g(z_2)=\overline{g(z_1)}$  
   and $g(z_1)=\overline{g(z_2)}$, hence $$f(z_1)=h(z_1)+\overline{g(z_1)}=h(z_2)+\overline{g(z_2)}=f(z_2).$$ This shows that $f_{a,\lambda}$ is not univalent when $a>2$. 

    The univalence for the case $-2\leq a<0$ is similar to that of $0<a\leq2$, and the non-univalence for the case 
    $a<-2$ is similar to that of $a>2$.
    We omit the details.
    The proof of Theorem \ref{T1} is thus completed.
\end{proof}  

\subsection{Pre-Schwarzian and Schwarzian norms}
In what follows, we derive the pre-Schwarzian and Schwarzian norms of the family $\mathcal{K}_{a,\lambda}$.
\begin{theorem}\label{T2}
	If $f=h+\bar{g}\in\mathcal{K}_{a,\lambda}$, then 
    \begin{equation}\label{31}
        P_f(z)=\dfrac{2(z+a)}{1-z^2}+\dfrac{\lambda}{1-\lambda z}+\dfrac{\lambda^2\bar{z}}{1-\abs{\lambda z}^2}
    \end{equation}
    and
     \begin{equation}\label{32}
       \|P_f(z)\|\leq2\left(1+\abs{a}\right)+2\lambda^2+\lambda.
    \end{equation}
   
\end{theorem}
\begin{proof}
	We begin by computing the derivative of the generalized Koebe function
\begin{equation}\label{T2-3} 	  k_a'(z)=\dfrac{{(1+z)}^{a-1}}{{(1-z)}^{a+1}}.
	\end{equation}
    For $f=h+\overline{g}\in\mathcal{K}_{a,\lambda}$, we have 
	 \begin{equation}\label{T2-4}
	    h'(z)=\dfrac{k_a'(z)}{1-\lambda z}.
	\end{equation}
    Differentiating \eqref{T2-4} logarithmically, it yields
    \begin{equation}\label{T2-5}
        \dfrac{h''(z)}{h'(z)}=\dfrac{2(z+a)}{1-z^2}+\dfrac{\lambda}{1-\lambda z}.
    \end{equation}
     By calculation, the assertions \eqref{31} and \eqref{32} of Theorem \ref{T2} follow from \eqref{T2-5}.
\end{proof}

\begin{theorem}\label{T3}
	If $f=h+\bar{g}\in\mathcal{K}_{a,\lambda}$, then 
    \begin{equation}\label{T3-1}
        \begin{aligned}
S_f(z)=&\dfrac{2\left(1-a^2\right )}{{\left(1-z^2\right)}^2}+\dfrac{\lambda^2}{2\left(1-\lambda z\right)^2}-\dfrac{2\lambda(z+a)}{\left(1-z^2\right)(1-\lambda z)}\\
	&+\dfrac{\lambda^2\overline{z}[-3\lambda z^2+2(1-a\lambda )z+2a+\lambda]}{(1-\lambda^2|z|^2)(1-z^2)(1-\lambda z)}+\dfrac{3\lambda^4\overline{z}^2}{\left(1-\lambda^2|z|^2\right)^2}
        \end{aligned}
    \end{equation}
    and
    \begin{equation}\label{T3-2}
        \|S_f(z)\|\leq\lambda^4+2\lambda^3(\abs{a}+1)+\lambda^2\left(4\abs{a}+\dfrac{13}{2}\right)+2\lambda(\abs{a}+2)+2\abs{1-a^2}.
    \end{equation}
\end{theorem}
\begin{proof}
     By \eqref{4} and \eqref{1}, we have
        \begin{equation}\label{T3-3}
            k_a''(z)=\dfrac{2(a+z)}{(1-z)^4}\left(\dfrac{1+z}{1-z}\right)^{a-2},
        \end{equation}
        \begin{equation}\label{T3-4}
            k_a'''(z)=\dfrac{2\left(1+z\right)^{a-3}}{\left(1-z\right)^{a+3}}\left(3z^2+6az+2a^2+1\right)
        \end{equation}
        and
         \begin{equation}\label{T3-5}
            h'''(z)=\dfrac{k_a'''(z)(1-\lambda z)^2+2\lambda\left[k_a''(z)(1-\lambda z)+\lambda k_a'(z)\right]}{(1-\lambda z)^3}.
        \end{equation}
  It follows from \eqref{T2-3}, \eqref{T2-5}, \eqref{T3-3}, \eqref{T3-4}, and \eqref{T3-5} that
    \begin{equation}\label{T3-6}
        \begin{aligned}
           S_f(z)=&\dfrac{k_a'''(z)}{k_a'(z)}+\dfrac{2\lambda k_a''(z)}{(1-\lambda z)k_a'(z)}+\dfrac{2\lambda^2}{(1-\lambda z)^2}-\dfrac{3}{2}\left(\dfrac{k_a''(z)}{k_a'(z)}+\dfrac{\lambda }{1-\lambda z}\right)^2\\
    &+\dfrac{\lambda^2\bar{z} }{1-\abs{\lambda z}^2}\left(\dfrac{k_a''(z)}{k_a'(z)}+\dfrac{\lambda }{1-\lambda z}\right)-\dfrac{3}{2}\left(\dfrac{\lambda^2\bar{z} }{1-\abs{\lambda z}^2}\right)^2.
        \end{aligned}
    \end{equation}
By \eqref{T3-6}, we deduce that \eqref{T3-1} and \eqref{T3-2} hold. 
   \end{proof} 

   \begin{rem}
{\rm By setting $a=2$ and $\lambda=0$ in Theorems \ref{T2} and \ref{T3}, they coincide with the classical results
 involving pre-Schwarzian and Schwarzian norms of univalent analytic functions (cf. \cite{b}).}
\end{rem}

\subsection{Coefficient estimates}
We now provide the estimates of the initial coefficients of $f\in\mathcal{K}_{a,\lambda}$.  
\begin{theorem}\label{T4}
	If $f=h+\bar{g}\in\mathcal{K}_{a,\lambda}$, then 
    $$|a_2|\leq|a|+\frac{\lambda}{2},$$ 
	$$|a_3|\leq\dfrac{1}{3}\left(\lambda^2+2\abs{a}\lambda+2a^2+1\right),$$
	$$|a_4|\leq\dfrac{1}{3}\abs{a}^3+\dfrac{2}{3}\abs{a}+
        \dfrac{1}{4}\lambda\left(\lambda^2+2\abs{a}\lambda+2a^2+1\right),$$
    $$|b_3|\leq\dfrac{1}{3}\lambda^2+\dfrac{2}{3}\abs{a}\lambda$$
and $$|b_4|\leq\dfrac{1}{4}\lambda\left(\lambda^2+2\abs{a}\lambda+2a^2+1\right).$$
All of these inequalities are sharp.
\end{theorem}
\begin{proof} Suppose that $f=h+\bar{g}\in\mathcal{K}_{a,\lambda}$.
	It follows from \eqref{1} that
    \begin{equation}\label{T4-1}
        \left(1-b_1\right)z+\sum\limits_{n=2}^{\infty}(a_n-b_n)z^n=k_a(z)
    \end{equation}
    and
    \begin{equation}\label{T4-2}
        \sum\limits_{n=2}^{\infty}nb_nz^{n-1}=\sum\limits_{n=2}^{\infty}\lambda na_nz^n.
    \end{equation}
	By \eqref{T2-3}, \eqref{T3-3}, and \eqref{T3-4}, we get 
    $$\dfrac{k_a''(0)}{2!}=a,$$ 
    \begin{equation*}\label{T4-3}
        \dfrac{k_a'''(0)}{3!}=\dfrac{2}{3}a^2+\dfrac{1}{3}
    \end{equation*}
    and 
     \begin{equation*}\label{T4-4}
        \dfrac{{k_a}^{(4)}(0)}{4!}=\dfrac{1}{3}a^3+\dfrac{2}{3}a.
    \end{equation*}
	By comparing the coefficients of $z^2$, $z^3$ and $z^4$ on both sides of \eqref{T4-1}, we obtain
	\begin{equation}\label{T4-5}
		\begin{cases}
			a_2-b_2=a,\\
			a_3-b_3=\frac{2}{3}a^2+\frac{1}{3},\\
			a_4-b_4=\frac{1}{3}a^3+\frac{2}{3}a.
		\end{cases}
	\end{equation}
	Similarly, it follows from \eqref{T4-2} that
	\begin{equation}\label{T4-6}
		\begin{cases}
			2b_2=\lambda,\\
			3b_3=2\lambda a_2,\\
			4b_4=3\lambda a_3.
		\end{cases}
	\end{equation}
   Combining \eqref{T4-5} and \eqref{T4-6} yields
    \begin{equation*}\label{T4-7}
        a_2=a+\dfrac{\lambda}{2},
    \end{equation*}
     \begin{equation*}\label{T4-8}
        a_3=\dfrac{1}{3}\left(\lambda^2+2a\lambda+2a^2+1\right),
    \end{equation*}
     \begin{equation*}\label{T4-9}
        a_4=\dfrac{1}{3}a^3+\dfrac{2}{3}a+
        \dfrac{1}{4}\lambda\left(\lambda^2+2a\lambda+2a^2+1\right),
    \end{equation*}
  
     \begin{equation*}\label{T4-10}
        b_3=\dfrac{1}{3}\lambda^2+\dfrac{2}{3}a\lambda
    \end{equation*}  and 
     \begin{equation*}\label{T4-11}
        b_4=\dfrac{1}{4}\lambda\left(\lambda^2+2a\lambda+2a^2+1\right).
    \end{equation*}
The inequalities in Theorem \ref{T4} follow directly from the above equations.
\end{proof}

\begin{rem}
{\rm By setting $a=2$ and $\lambda\rightarrow 1^-$ in Theorem \ref{T4}, the coefficients estimates coincide with the harmonic Koebe function introduced by Clunie and Sheil-Small
 \cite{cs}.}
\end{rem}

\subsection{Growth and area theorems for the class $\mathcal{K}_{a,\lambda}$}
Finally, we derive the growth and area theorems for the class $\mathcal{K}_{a,\lambda}$.
\begin{theorem}\label{T5}
	Let $f=h+\bar{g}\in\mathcal{K}_{a,\lambda}$. Then
    \begin{equation}\label{T5-1}
    \begin{split}
     \displaystyle\int_{0}^{r}\dfrac{(1-\lambda \rho)(1-\rho)^{a-1}}{(1+\lambda \rho)(1+\rho)^{a+1}}\,d\rho\leq|f(z)|\leq \displaystyle\int_{0}^{r}\dfrac{(1+\lambda \rho)(1+\rho)^{a-1}}{(1-\lambda \rho)(1-\rho)^{a+1}}\,d\rho\quad(a\geq 1),
    \end{split}
    \end{equation}

 \begin{equation}\label{T5-7}
    \begin{aligned}
     \displaystyle\int_{0}^{r}\dfrac{1-\lambda \rho}{(1+\lambda \rho)(1+\rho)^{2}}\,d\rho \leq|f(z)|\leq   
     \displaystyle\int_{0}^{r}\dfrac{1+\lambda \rho}{(1-\lambda \rho)(1-\rho)^{2}}\,d\rho\quad(-1<a<1)
    \end{aligned}
    \end{equation}
and
     \begin{equation}\label{T5-2}
    \begin{aligned}
     \displaystyle\int_{0}^{r}\dfrac{(1-\lambda \rho)(1+\rho)^{a-1}}{(1+\lambda \rho)(1-\rho)^{a+1}}\,d\rho \leq|f(z)|\leq   
     \displaystyle\int_{0}^{r}\dfrac{(1+\lambda \rho)(1-\rho)^{a-1}}{(1-\lambda \rho)(1+\rho)^{a+1}}\,d\rho
     \quad(a\leq -1).
    \end{aligned}
    \end{equation}
\end{theorem}
\begin{proof}
Let $f=h+\bar{g}\in\mathcal{K}_{a,\lambda}$ and $|\xi|=r\leq 1$.  By \eqref{T2-3} and \eqref{T2-4}, we get 
\begin{equation}\label{T5-3}
   \dfrac{(1-r)^{a-1}}{(1+\lambda r)(1+r)^{a+1}}\leq|h'(\xi)|\leq \dfrac{(1+r)^{a-1}}{(1-\lambda r)(1-r)^{a+1}}\quad(a\geq 1),
\end{equation}

\begin{equation}\label{T5-8}
\dfrac{1}{(1+\lambda r)(1+r)^{2}} \leq|h'(\xi)|\leq   
     \dfrac{1}{(1-\lambda r)(1-r)^{2}}\quad(-1<a<1)
\end{equation}
and
\begin{equation}\label{T5-4}
\dfrac{(1+r)^{a-1}}{(1+\lambda r)(1-r)^{a+1}} \leq|h'(\xi)|\leq   \dfrac{(1-r)^{a-1}}{(1-\lambda r)(1+r)^{a+1}}\quad(a\leq -1).
\end{equation}

For $a\geq 1$, let $\Gamma$ be the line segment joining $0$ and $z$.
Then 
\begin{equation}\label{T5-5}
    \begin{aligned}
        |f(z)|&=\left|\displaystyle\int_{\Gamma}\dfrac{\partial f}{\partial \xi}\,d\xi+\dfrac{\partial f}{\partial \bar{\xi}}\,d\bar{\xi}\right|\\
        &\leq\displaystyle\int_{\Gamma}\left(|h'(\xi)|+|g'(\xi)|\right)\,|d\xi|\\
        &=\displaystyle\int_{\Gamma}(1+\lambda|\xi|)|h'(\xi)|\,|d\xi|\\
        & \leq\displaystyle\int_{0}^{r}\dfrac{(1+\lambda \rho)(1+\rho)^{a-1}}{(1-\lambda \rho)(1-\rho)^{a+1}}\,d\rho.\\
    \end{aligned}
\end{equation}
Moreover, let $\widetilde{\Gamma}$  be the preimage of the line segment joining $0$ and $f(z)$ under $f$. Then
\begin{equation}\label{T5-6}
    \begin{aligned}
        |f(z)|&=\left|\displaystyle\int_{\widetilde{\Gamma}}\dfrac{\partial f}{\partial \xi}\,d\xi+\dfrac{\partial f}{\partial \bar{\xi}}\,d\bar{\xi}\right|\\
        &\geq\displaystyle\int_{\widetilde\Gamma}\left(|h'(\xi)|-|g'(\xi)|\right)\,|d\xi|\\
        &=\displaystyle\int_{\widetilde\Gamma}(1-\lambda|\xi|)|h'(\xi)|\,|d\xi|\\
        &\geq\displaystyle\int_{0}^{r}\dfrac{(1-\lambda \rho)(1-\rho)^{a-1}}{(1+\lambda \rho)(1+\rho)^{a+1}}\,d\rho.\\
    \end{aligned}
\end{equation}
By \eqref{T5-5} and \eqref{T5-6}, we conclude that the assertion \eqref{T5-1} of Theorem \ref{T5} holds. 

Similarly, by
\eqref{T5-8} and \eqref{T5-4}, we get \eqref{T5-7} and \eqref{T5-2} for the cases $-1<a<1$ and $a\leq -1$, respectively.
\end{proof}

\begin{rem}{\rm
 By setting $a=2$ and $\lambda\rightarrow 1^-$ in Theorem \ref{T5}, it reduces to the classical
 growth theorem of univalent harmonic mappings in
 \cite[Theorem 1]{sh}.}
\end{rem}


Let $\mathcal{A}(f(\D_r))$ denote the area of $f(\D_r)$, where $\D_r:=r\D$ for $0<r<1$.
We prove the area theorem for the family $\mathcal{K}_{a,\lambda}$.
\begin{theorem}\label{T7}
	Let $f=h+\bar{g}\in\mathcal{K}_{a,\lambda}$. Then
     \begin{equation}\label{T7-2}
    \begin{split}
     2\pi\displaystyle\int_{0}^{r}\dfrac{(\rho-\lambda^2 \rho^3)(1-\rho)^{2(a-1)}}{(1+\lambda \rho)^2(1+\rho)^{2(a+1)}}\,d\rho\leq\mathcal{A}(f(\D_r))\leq 2\pi \displaystyle\int_{0}^{r}\dfrac{(\rho-\lambda^2 \rho^3)(1+\rho)^{2(a-1)}}{(1-\lambda \rho)^2(1-\rho)^{2(a+1)}}\,d\rho
    \end{split}
    \end{equation}
     for $a\geq 1$,

 \begin{equation}\label{T7-3}
    \begin{aligned}
     2\pi \displaystyle\int_{0}^{r}\dfrac{\rho-\lambda^2 \rho^3}{(1+\lambda \rho)^2(1+\rho)^{4}}\,d\rho \leq\mathcal{A}(f(\D_r))\leq   
     2\pi \displaystyle\int_{0}^{r}\dfrac{\rho-\lambda^2 \rho^3}{(1-\lambda \rho)^2(1-\rho)^{4}}\,d\rho
    \end{aligned}
    \end{equation}
     for $-1<a<1$,
and
     \begin{equation}\label{T7-4}
    \begin{aligned}
     2\pi \displaystyle\int_{0}^{r}\dfrac{(\rho-\lambda^2 \rho^3)(1+\rho)^{2(a-1)}}{(1+\lambda \rho)^2(1-\rho)^{2(a+1)}}\,d\rho \leq\mathcal{A}(f(\D_r))\leq   
     2\pi \displaystyle\int_{0}^{r}\dfrac{(\rho-\lambda^2 \rho^3)(1-\rho)^{2(a-1)}}{(1-\lambda \rho)^2(1+\rho)^{2(a+1)}}\,d\rho
    \end{aligned}
    \end{equation}
    for $a\leq -1$.
\end{theorem}
\begin{proof}
Suppose that $f=h+\bar{g}\in\mathcal{K}_{a,\lambda}$. Then
\begin{equation}\label{T7-1}
    \begin{aligned}
        \mathcal{A}(f(\D_r))&=\displaystyle\iint_{\D_r}\left(|h'(\xi)|^2-|g'(\xi)|^2\right)\,dx\,dy\\
        &=\displaystyle\iint_{\D_r}\left(1-\lambda^2|\xi|^2\right)|h'(\xi)|^2\,dx\,dy.
    \end{aligned}
\end{equation}
By \eqref{T5-3}, \eqref{T5-8}, \eqref{T5-4}, and \eqref{T7-1}, we obtain the assertions of Theorem \ref{T7}.
\end{proof}

\begin{rem}
{\rm By setting $a=2$ and $\lambda\rightarrow 1^-$ in Theorem \ref{T7}, it coincides with the area theorem in 
 \cite[p. 90]{du}.}
\end{rem}


	
	\vskip.05in
	\noindent{\bf Funding.}
Z.-G. Wang was partially supported by the \textit{Key Project of Education Department of Hunan Province}, and
the \textit{Natural Science Foundation of Changsha} under Grant no. kq2502003
of the P. R. China. J.-L. Qiu was partially supported by the \textit{Graduate Research Innovation Project of Hunan Province} under Grant no. CX20240803 of the P. R. China.
A. Rasila was partially supported by the \textit{Natural Science Foundation of Guangdong Province} under Grant no. 2024A1515010467 of the P. R. China, and the \textit{Li Ka Shing Foundation GTIIT-STU Joint Research Grant} under Grant no. 2024LKSFG06. 

\vskip.05in
\noindent{\bf Acknowledgements.}
The authors thank Shengyu Li, Zhi-Hong Liu and Xiao-Yuan Wang for their stimulating
discussions in the preparation of this paper. They also thank the referees and the editor for their valuable
comments and suggestions, which were essential in improving the quality
of this paper.

	\vskip.05in
    
	\noindent{\bf Conflicts of interest.} The authors declare that they have no conflict of interest.
	
	\vskip .05in
\noindent{\bf Data availability statement.}  Data sharing is not applicable to this article as no datasets were generated or analysed during the current study.

\end{document}